\documentclass[12pt]{article}%
\usepackage{amsfonts}
\usepackage{amsmath}
\usepackage{amssymb}
\usepackage{graphicx}%
\providecommand{\U}[1]{\protect\rule{.1in}{.1in}}

\newtheorem{theorem}{Theorem}

\newtheorem{algorithm}[theorem]{Algorithm}

\newtheorem{definition}[theorem]{Definition}
\newtheorem{example}[theorem]{Example}

\newtheorem{lemma}[theorem]{Lemma}

\newtheorem{proposition}[theorem]{Proposition}

\newenvironment{proof}[1][Proof]{\noindent\textbf{#1.} }{\ \rule{0.5em}{0.5em}}
\begin{document}

\title{The Schubert normal form of a 3-bridge link and the 3-bridge link group}
\author{Margarita Toro and Mauricio Rivera\\Universidad Nacional de Colombia, Medell\'{\i}n, Colombia.\\mmtoro@unal.edu.co and mrivera@unal.edu.co}
\date{November 2016}
\maketitle

\begin{abstract}
We introduce the Schubert form a $3$-bridge link diagram, as a generalization
of the Schubert normal form of a $3$-bridge link. It consists of a set of six
positive integers, written as $\left(  p/n,q/m,s/l\right)  $, with some
conditions and it is based on the concept of $3$-butterfly. Using the Schubert
normal form of a $3$-bridge link diagram, we give two presentations of the
3-bridge link group. These presentations are given by concrete formulas that
depend on the integers $\left\{  p,n,q,m,s,l\right\}  .$ The construction is a
generalization of the form the link group presentation of the $2$-bridge link
$p/q$ depends on the integers $p$ and $q$.

\end{abstract}

\section{Introduction}

In \cite{HMTT4} it was introduced the butterfly presentation of a link diagram
as a generalization of the 2-bridge Schubert's notation. Moreover, the
particular concept of a $3$-butterfly was implemented in order to study
$3$-bridge links and to obtain a codification of a $3$-bridge link diagram.
Here, for our purpose, we do not need all the machinery of the butterfly
construction presented in \cite{HMTT1}, so we will take a different approach.
We will describe the construction of the codification by a direct and
combinatorial approach, using the ideas in \cite{Fe}, where Ferri constructed
the crystallization of the double cover of $S^{3}$, with a link as
ramification set. For any link diagram $L$ there is a strong relation between
crystallization of the double cover of $S^{3}$, with $L$ as the ramification
set, and the $3$-butterfly associated to $L$, that we will explain in
\cite{RiTo}. For any $n$-bridge link diagram the construction of an
$n$-butterfly is possible, see \cite{HMTT1}, but in this paper we want to be
specific and we will work only with $3$-bridge link diagrams.

To a 3-bridge diagram we associate a 3-butterfly that is described by a set of
six positive integers $\left\{  p,n,q,m,s,l\right\}  $, with some
restrictions, and then we define the \textit{Schubert form} of the link
diagram as $\left(  p/n,q/m,s/l\right)  ,$ for geometrical reasons that will
be explained in Section 1. As each 3-bridge link admits infinitely many
different link diagrams, the \textit{Schubert normal form }for a link $L$ is
defined by taking the minimum among all 3-bridge link diagrams of $L$
according to a lexicographical type of order, see \cite{HMTT4}. For the
purpose of this paper we only need the Schubert form of the link diagram, but
in further research and in the compilation of link tables, it will be
interesting to consider the \textit{Schubert normal form }of a link.

In this paper we find formulas for the over and under presentation of the
3-bridge link represented by the Schubert form $\left(  p/n,q/m,s/l\right)  $,
that depends on the integers $\left\{  p,n,q,m,s,l\right\}  $. The formula for
the under presentation of the 3-bridge link is a natural extension of the
formula for the presentation of the 2-bridge link $p/q$, that depends on the
integers $p$ and $q$.

The paper is organized as follows: in Section 1 we describe the construction
of a 3-butterfly associated to a 3-bridge link diagram $L$ and introduce the
Schubert form of $L$, that consists of a set of 6 positive integers, $\left(
p,n,q,m,s,l\right)  $,\ that captures the relevant information of the
3-butterfly and, therefore, of the 3-bridge diagram $L$. \ In Section
\ref{diagram} we describe a canonical diagram associate to a 3-butterfly
$\left(  p,n,q,m,s,l\right)  $, in a similar way to the canonical diagram of a
2-bridge link, see \cite{Sch}. In Section \ref{secorien} we will give an
orientation to this canonical diagram.

In Section \ref{secpermu} we define two permutations, $\gamma$ and $\phi$,
associated to the Schubert form $\left(  p/n,q/m,s/l\right)  $, and study the
composition $\mu=\gamma\phi$. The cyclic structure of $\mu$ is the key point
in the rest of the paper. A variation of the permutation $\mu$ is very useful
for the construction of a Gauss code for the link diagram, and, from there we
can find the Dowker code and we are able to compute link invariants, such as
the link group, the Seifert matrix, the Alexander, Jones and HOMFLY polynomials.

In Section \ref{secgroup}\ we present our main result, Theorems \ref{presover}%
, \ref{presoverpal} and \ref{presunder}, that give explicit presentations for
the knot group $\pi\left(  L\right)  $ of a $3$-bridge link $L$. These
presentation are described by clear algorithms, that are easy to program in a
computer and depend on the integers in the Schubert form of the link diagram.
In the last section we propose a special family of links, $\left(
p/n,p/n,p/n\right)  $, that have a strong symmetry that is reflected in the
group presentation. This symmetry could be exploited in the study of the
representations into SL$\left(  2,\mathbb{C}\right)  $ of the link group.

Some authors allow that any $n$-bridge link diagram can be consider as a
$k$-bridge link diagram, for any $k>n$, by considering bridges without any
undercrossings, see \cite{Neg1}. We neither allow this situation nor consider
a split link diagram with more than 3 components, as the one in Fig.
\ref{fig3}a, as a 3-bridge link diagram We work with 3-bridge link diagrams as
in \cite{BuZi} and \cite{Mur}.

\textbf{Remark on notation:} In \cite{Ri} the author uses subindexes and
denote a butterfly by $\left(  M_{1},N_{1},M_{2},N_{2},M_{3},N_{3}\right)  $.
In this paper we avoid the use of subindexes in the Schubert form, and prefer
to assign a different role to each integer, in that way we reach simpler formulas.

\section{Description of the 3-butterfly of a 3-bridge link
diagram\label{secdescrip}}

Let $L$ be a 3-bridge link such that the projection on the $xy$ plane is a
3-bridge diagram $D$. Let $a,b$ and $c$ the bridge projections. Draw an
ellipse around each of the bridges, in such a way that they are disjoint, and
they have the bridges as principal axes. Each ellipse will intercept the
diagram $D$ in an even number of points, that will be the vertices. We denote
by $P,Q$ and $S$ the ellipses around the bridges $a,b$ and $c$, respectively,
and let $2p\ $(resp. $2q$ and $2s$) the number of intersections of $P$ (resp.
$Q$ and $S)$ with the diagram $D$. Following \cite{HMTT1}, the ellipses
$P,Q,S$ are called butterflies.

Take the graph $R_{1}$ formed by the butterflies $P,Q,S$, the vertices and the
bridges. In each butterfly we have the bridge, that divides each butterfly in
two halves, that will be the wings. The reflection along the bridges inside
the butterflies is called $\gamma$. The segment of the underarcs that are
inside the butterflies are forgotten, but they can be recovered with the
reflection $\gamma$. The edges outside the butterflies will give the
information on how the butterflies intercept to each other, see Fig.
\ref{fig1}a. Each one of these edges connect two vertices of two butterflies,
we identify these vertices, to form a set that will be the vertices of our
graph. The identification will give an involution on the vertex of the graph,
that we call $\phi$.%

\begin{figure}
[ht]
\begin{center}
\includegraphics[
height=1.6077in,
width=4.6363in
]%
{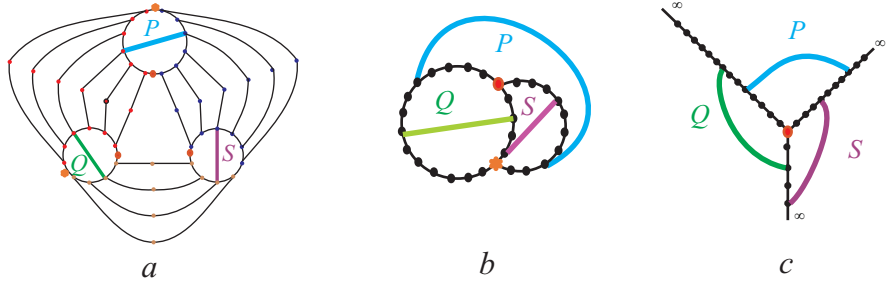}%
\caption{Graph for 3-butterflies.}%
\label{fig1}%
\end{center}
\end{figure}

We define the $3$-\textit{butterfly} as the graph $R=R_{1}/\phi$, formed with
the vertices of $R_{1}\ $identified by the involution $\phi$. We draw the
graph $R$ in any of the three forms shown in Fig. \ref{fig1}. If we consider
that the diagram $L$ is in $S^{2}=\partial B$, then the graph $R$ define a
polygonalization of $S^{2}$ formed by three polygons, as shown in Fig.
\ref{fig2}. Compare this simple construction with the formal one given in
\cite{HMTT1} and \cite{HMTT4}. When we identify the butterflies, there will
appear two new points, that will be denoted $0\ $and$\ \ast$. These two points
are fundamental, but they are not considered vertices of the $3$-butterfly.
There can be only two basic forms for the graph $R$, that are determined by
the way the butterflies $P,Q$ and $S$ intersect.%

\begin{figure}
[ht]
\begin{center}
\includegraphics[
height=1.5428in,
width=3.0995in
]%
{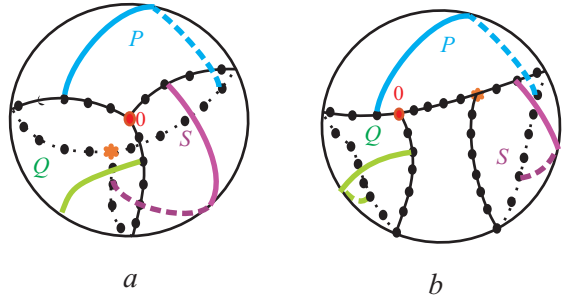}%
\caption{Types of $3$-butterflies.}%
\label{fig2}%
\end{center}
\end{figure}

\noindent\textbf{Type I}: The three butterflies intersect in the two points
$0\ $and$\ \ast$, see Fig. \ref{fig2}.a.

\noindent\textbf{Type II}: Two of the polygons do not intersect, see Fig.
\ref{fig2}.b. When a 3-bridge link diagram produces a type II butterfly, there
is a wave move, see \cite{Neg1} that allows us to construct a new 3-bridge
diagram with lower crossing number. So, we work only with type I butterflies,
such that there are no wave moves.

In order to obtain a canonical way to describe a $3$-butterfly, we will always
assume that
\begin{equation}
p\geq q\geq s\geq2, \label{restr0}%
\end{equation}
the condition $s\geq2$ is to ensure that each bridge has at least one
crossing. By rotating the plane and interchanging the points $0$ and $\ast$,
we can always obtain a $3$-butterfly diagram with $P$ at the top, $Q$ to the
left and $S$ to the right, and we read it in the counterclockwise direction,
$P\ Q\ S$, as shown in Fig. \ref{fig1}.

Let $\left\vert P\cap Q\right\vert =t\ $be the number of vertices between $P$
and $Q$,$\ v=\left\vert Q\cap S\right\vert $ and\ $w=\left\vert P\cap
S\right\vert $. As each butterfly intersects the other two, then $t,v$ and $w$
are positive integers that satisfy$\ 2p=t+w,\ 2q=t+v,\ $and $2s=v+w,$therefore%
\begin{equation}
t=p+q-s,\ \ \ v=q+s-p,\ \ \ w=p+s-q. \label{internum}%
\end{equation}

As we will only consider the link diagrams with $v\geq1$, then
\begin{equation}
p+1\leq s+q. \label{restr1}%
\end{equation}
So the integers $p,q$ and $s$ satisfy (\ref{restr0}) and (\ref{restr1}).
Reciprocally, if we have integers satisfying (\ref{restr0}) and (\ref{restr1})
we can construct the butterflies $P,Q$ and $S$.

Now let us describe the positions of the bridges. We orient clockwise each
butterfly. From the point $0$, following the orientation, we count the number
of vertices between $0$ and the vertex in which the bridge begins. We call
$n,m$ and $l$ the initial points of the bridges in $P,Q$ and $S$,
respectively. Clearly
\begin{equation}
1\leq n\leq p,\ \ \ \ \ 1\leq m\leq q,\ \ \ \ \ 1\leq l\leq s. \label{restr2}%
\end{equation}

We are working only with link diagrams with exactly three bridges and not only
two or one, this impose conditions on the integers $n,m$ and $l$. In
\cite{HMTT4} they found the conditions given in the following theorem.

\begin{theorem}
\label{teoclassi}Every $3$-butterfly\textit{\ defines a unique set of integers
}$\left\{  p,m,q,n,s,l\right\}  $\textit{\ such that }%
$\ \ \ \ \ \ \ \mathit{\ }$
\begin{align}
p &  \geq q\geq s\geq2,\ 1\leq n\leq p,\ \ \ \ 1\leq m\leq q,\ \ \ \ 1\leq
l\leq s,\ \ \ \ p+1\leq s+q\label{resteo}\\
n+m &  \not =q+1,\ \ \ \ n+l\not =p+1,\ \nonumber\\
\text{if}\ \ m &  >q+s-p\text{ then }n+m\not =2q+1\text{ and }n+m\not =%
2q-p+1\ \nonumber\\
\text{if }l &  <p-s\text{ then }n+l\not =p-s+1\text{ }\nonumber\\
\text{if }m &  \leq q+s-p\text{ then\ }m+l\not =s+1.\nonumber
\end{align}

Reciprocally, if a set of integers $\left\{  p,m,q,n,s,l\right\}  $ satisfies
conditions (\ref{resteo}) then it defines a $3$-butterfly.
\end{theorem}

Note that we may think that inside the butterfly $P$ (resp. $Q,$ $S$) the
bridge $a$ make a $\left(  n/p\right)  \pi$ rotation, (resp. $\left(
m/q\right)  \pi,\left(  l/s\right)  \pi$). For this geometrical reason we want
to use the notation $\left(  p/n,q/m,s/l\right)  $ instead of $\left\{
p,m,q,n,s,l\right\}  $, but $p/n$ is not considered as a rational number.

The conditions on the integers $\left\{  p,n,q,m,s,l\right\}  $ impose in
(\ref{resteo}) define a $3$-butterfly and a link diagram $L$, but it is
possible that $L$ is a split link with some trivial components, as the diagram
associated to $\left\{  5,1,5,2,5,1\right\}  $ shows, see Fig. \ref{fig3}.%

\begin{figure}
[ptb]
\begin{center}
\includegraphics[
height=1.4096in,
width=4.4183in
]%
{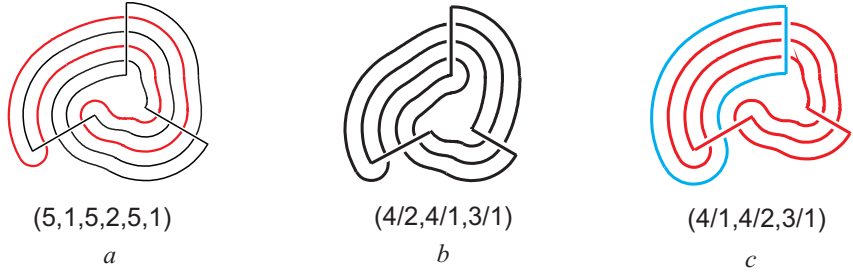}%
\caption{Diagrams associated to a non reduced butterfly and to two Schubert
forms}%
\label{fig3}%
\end{center}
\end{figure}

\begin{definition}
We say that $\left(  p,n,q,m,s,l\right)  $ is a 3-butterfly if the set of
integers $\left\{  p,n,q,m,s,l\right\}  $ satisfies the conditions
(\ref{resteo}). We say that a 3-butterfly is reduced if the associated diagram
is a 3-bridge diagram without any trivial components. We say that $\left(
p/n,q/m,s/l\right)  $ is a Schubert form of a 3-bridge link if the 3-butterfly
$\left(  p,n,q,m,s,l\right)  $ is reduced.
\end{definition}

We need to be careful with the relative order of the numbers in
$(p/n,q/m,s/l)$.

\begin{example}
If we change the order, it is possible that we get different Schubert forms.
The Schubert form $(4/2,4/1,3/1)\ $represents a knot and $\left(
4/1,4/2,3/1\right)  $ represents a two component link. See Fig. \ref{fig3}.
\end{example}

\section{Algorithm to draw a canonical 3-bridge link diagram\label{diagram}%
}

We associate to each 3-butterfly $\left(  p,n,q,m,s,l\right)  $ a canonical
diagram, in a similar way as the canonical diagram of a $2$-bridge link is
associated to $p/q$, see \cite{Kawa}. We draw the three bridges as
three segments: bridge $a$ as a vertical segment; bridge $b$ as a segment
forming a $120^{0}$ angle with the bridge $a$ and bridge $c$ as a segment
forming a $240^{0}$ angle with the bridge $a$.

We divide the bridge $a$ in $p$ segments, and we fix two points in each
division, one to the left and one to the right, except at the extreme points,
where there is only one. Label them with $A=\left\{  a_{0},a_{1}%
,\cdots,a_{2p-1}\right\}  $, in a counterclockwise sense, so the extreme
bridges are labeled $a_{0}$ and $a_{p}$. For the bridge $b$ we repeat the
process, but we divide the bridge in $q$ segments and label the points with
$B=\left\{  b_{0},\cdots,b_{2q-1}\right\}  $. For the bridge $c$ the number of
segments is $s$ and the labels are $C=\left\{  c_{0},\cdots,c_{2s-1}\right\}
$. The subscripts of $A\ $(resp. $B$ and $C$) are taken $\operatorname{mod}%
\left(  2p\right)  $, (resp. $\operatorname{mod}\left(  2q\right)  $ and
$\operatorname{mod}\left(  2s\right)  $).

To draw the link diagram we need to join, with appropriate arcs, the points
$a_{i},b_{j}\ $and $c_{k}$, $i\in\mathbb{Z}_{2p},j\in\mathbb{Z}_{2q},$ and
$k\in\mathbb{Z}_{2s},$ according to the rules given by permutations $\phi$ and
$\gamma$.

There are $t=p+q-s$ arcs between the $a$ and $b$ bridges, namely
\[
a_{n-1}b_{m},a_{n-2}b_{m+1},a_{n-3}b_{m+2},\ldots,a_{n-j}b_{m+j-1}%
,\ldots,a_{n-t}b_{m+t-1},
\]
likewise there are $v=q+s-p$ arcs between the $b$ and $c$ bridges, that are%
\[
b_{m-1}c_{l},b_{m-2}c_{l+1},b_{m-3}c_{l+2},\ldots,b_{m-j}c_{l+j-1}%
,\ldots,b_{m-v}c_{l+v-1},
\]
and, finally, $w=p+s-q$ arcs between the $c$ and $a$ bridges,%
\[
c_{l-1}a_{n},c_{l-2}a_{n+1},c_{l-3}a_{n+2},\ldots,c_{l-j}a_{n+j-1}%
,\ldots,c_{l-w}a_{n+w-1}.
\]

It is enough to know how to construct the first arc between each pair of
bridges, and the rest of the arcs are "\textit{parallel}" arcs to them, see
Fig. \ref{fig4}%

\begin{figure}
[ht]
\begin{center}
\includegraphics[
height=1.0957in,
width=4.6544in
]%
{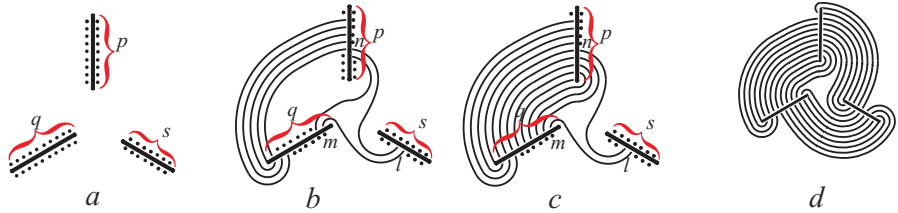}%
\caption{Drawing a canonical 3-bridge diagram.}%
\label{fig4}%
\end{center}
\end{figure}

In the rest of this paper we will refer to the diagram described as the link
canonical diagram associated to the Schubert form $\left(  p/n,q/m,s/l\right)
$. Notice that if $\left(  p/n,q/m,s/l\right)  $ does not satisfy the
conditions in Theorem \ref{teoclassi}, we still may use this algorithm to draw
a link diagram.

\begin{lemma}
If the Schubert form $\left(  p/n,q/m,s/l\right)  $ is reduced, the $3$-bridge
diagram has $p+q+s-3$ crossings.
\end{lemma}

\section{Permutations associated to a Schubert form \label{secpermu}}

The conditions for a 3-butterfly $\left\{  p,n,q,m,s,l\right\}  $ to be
reduced can not be described using simple conditions on the integers in a
similar way as the conditions to be a 3-butterfly given in (\ref{resteo}). Now
we need to go deeper and study in detail the permutations $\phi$ and $\gamma$.
Given a set $\left\{  p,n,q,m,s,l\right\}  $ that satisfies (\ref{resteo}) we
construct explicitly the associated $3$-butterfly and then we draw the
3-bridge diagram.

Define the 3-butterfly by labelling the vertices of each of the butterflies:
$P$ have vertices labeled by $A=\{a_{0},\cdots a_{i},\cdots,a_{2p-1}%
\},\ i\in\mathbb{Z}_{2p};$ $Q$ with vertices $B=\{b_{0},\cdots,b_{j}%
,\cdots,b_{2q-1}\},j\in\mathbb{Z}_{2q};$ and $S$ with vertices $C=\{c_{0}%
,\cdots,c_{l},\cdots,c_{2s-1}\},$ $l\in\mathbb{Z}_{2s}.$ The bridge ends are
labeled by $a_{0}$ and $a_{p}$ in $P$ (resp. by $b_{0}$ and $b_{q}$ in $Q$ and
$c_{0}$ and $c_{s}$ in $S$). See Fig. \ref{fig5}.%

\begin{figure}
[ptb]
\begin{center}
\includegraphics[
height=2.1421in,
width=4.6406in
]%
{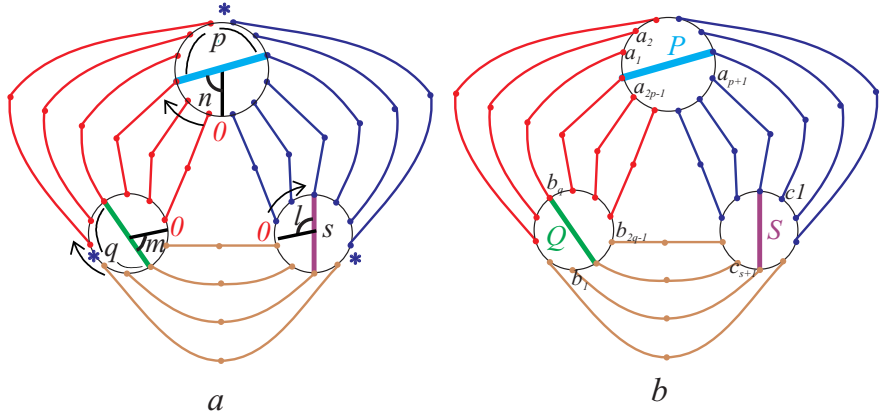}%
\caption{Orientation and labels of a $3$-butterfly.}%
\label{fig5}%
\end{center}
\end{figure}

We have the permutations $\gamma$ and $\phi\ $on the set $A\cup B\cup C$. The
permutation $\gamma$ is the reflection along the bridges. The permutation
$\phi$ is determined by the identification of the vertices of two butterflies,
so in the $3$-butterfly each vertex has two labels. The proofs of the
following lemmas are straightforward computations.

\begin{lemma}
\label{lemmagama}The function defined in the set $A\cup B\cup C$ by
\begin{equation}%
\begin{array}
[c]{ccc}%
\gamma\left(  a_{i}\right)  =a_{2p-i}, &  & 0\leq i<2p,\ \ \\
\gamma\left(  b_{j}\right)  =b_{2q-j}, &  & 0\leq j<2q,\ \ \\
\gamma\left(  c_{h}\right)  =b_{2s-h}, &  & 0\leq h<2s\ \
\end{array}
\label{permugama}%
\end{equation}
is an order 2 permutation. The set of fixed points is%
\begin{equation}
E=\left\{  a_{0},a_{p,}b_{0},b_{q},c_{0},c_{s}\right\}  \text{.}
\label{extrem}%
\end{equation}

\end{lemma}

The set $E=\left\{  a_{0},a_{p,}b_{0},b_{q},c_{0},c_{s}\right\}  $ corresponds
to the endpoints of the bridges. It will play an important role in the rest of
the paper.

\begin{lemma}
\label{lemmafi}The map $\phi:$ $A\cup B\cup C\rightarrow A\cup B\cup C$
defined by
\begin{equation}%
\begin{array}
[c]{llll}%
a_{n-i} & \longleftrightarrow & b_{m+i-1,} & \text{ if }1\leq i\leq t,\\
a_{n+j} & \longleftrightarrow & c_{l-j-1,} & \text{ if }0\leq j\leq w-1,\\
b_{m-h} & \longleftrightarrow & c_{l+h-1,} & \text{ if }1\leq h\leq v,
\end{array}
\label{permufi}%
\end{equation}
is an order 2 permutation, where $t=p+q-s,$ $v=q+s-p$ and $w=p+s-q$.
\end{lemma}

Note that $\phi$ does not have fixed points, and among the bicycles in $\phi$
there is no a bicycle in the set
\begin{align}
\mathcal{F}  &  =\{\left(  a_{0},b_{0}\right)  ,\left(  a_{0},b_{q}\right)
,\left(  a_{0},c_{0}\right)  ,\left(  a_{0},c_{s}\right)  ,\left(  b_{0}%
,a_{p}\right)  ,\left(  b_{0},c_{0}\right)  ,\label{forbiden}\\
&  \left(  b_{0},c_{s}\right)  ,\left(  c_{0},a_{p}\right)  ,\left(
c_{0},b_{q}\right)  ,\left(  a_{p},b_{q}\right)  ,\left(  a_{p},c_{s}\right)
,\left(  b_{q},c_{s}\right)  \}\nonumber
\end{align}

The construction of $\phi$ is well defined for any polygonalization of $S^{2}$
with $3$ polygons, even if they do not satisfy the conditions of Theorem
\ref{teoclassi}. In fact, in terms of the permutation $\phi$, we can rewrite
Theorem \ref{teoclassi} as follows.

\begin{theorem}
A set $\left\{  p,n,q,m,s,l\right\}  $, with $p\geq q\geq s\geq2,\ 1\leq n\leq
p,\ 1\leq m\leq q,\ 1\leq l\leq s\ $describes a 3-butterfly if and only if the
associated permutation $\phi$ does not have any of the bicycles in the set
$\mathcal{F}$.
\end{theorem}

We study the cyclic decomposition of $\mu=\phi\gamma$. The orbit of a vertex
$v$ will be denoted by $\mathcal{O}_{\mu}\left(  v\right)  $. For a cycle
$\tau=\left(  z_{1}\ z_{2}\ \cdots\ z_{k}\right)  $ we will use the same
symbol to refer to the cycle, to its orbit $\left\{  z_{1},z_{2},\cdots
,z_{k}\right\}  $ and to the word $z_{1}z_{2}\cdots z_{k}$. The length of
$\tau$ will be denoted by $|\tau|,\tau\left(  x\right)  $ will denote the
cycle that contains $x$ and for a function $\Gamma,$ $\Gamma\left(
\tau\right)  $ will be the word (set) formed by applying $\Gamma$ to each
element in $\tau$.

\begin{theorem}
\label{teored}Let $\left(  p,n,q,m,s,l\right)  $ be a 3-butterfly, $\gamma$
and $\phi$ be its associated permutations, given in Lemmas  \ref{lemmagama}
and \ref{lemmafi} and let $\mu=\phi\gamma$. $\left(  p/n,q/m,s/l\right)  $ is
a Schubert form for a 3-bridge link if and only if $\mu$ is the product of
three disjoint cycles,$\ \mu=\tau_{1}\tau_{2}\tau_{3}$ such that, for
\thinspace i=1,2,3, $|\tau_{i}\cap E|=2$, where $E$ is given in (\ref{extrem}).
\end{theorem}

\begin{proof}
Let $\mu=\phi\gamma$ associated to the 3-butterfly $\left(
p,n,q,m,s,l\right)  $. The 3-butterfly$\ \left(  p,n,q,m,s,l\right)  $ defines
a Schubert form $\left(  p/n,q/m,s/l\right)  $ if and only if the 3-butterfly
is reduced.

Suppose that the butterfly is reduced. The orbit of $a_{p}\ $under
$\mu,\ \mathcal{O}_{\mu}\left(  a_{p}\right)  $ will describe a path that
follows the underarc with initial point in $a_{p}$, so eventually it will
arrive to the endpoint of the underarc, say $e=\mu^{k}\left(  a_{p}\right)
$,\ $e\in$ $E,$ $E\ $defined in (\ref{extrem}). Then $\mu\left(  e\right)
=\phi\gamma\left(  e\right)  =\phi\left(  e\right)  =\phi\phi\gamma\mu
^{k-1}\left(  a_{p}\right)  =\gamma\mu^{k-1}\left(  a_{p}\right)  $, so the
orbit will go back to the same underarc, in opposite direction. We take
$\tau_{1}$ as the cycle formed by the orbit of $a_{p}$ and $\tau_{1}\cap E$
will contain exactly two vertices. We repeat the same process with the other
vertices in $E$. Since the butterfly is reduced, all the vertices will be
crossed by one of the underarcs, so we have only three orbits.

Reciprocally, if the butterfly is not reduced, there will be a component whose
vertex will not be in the orbit of any of the elements in $E$. See Fig.
\ref{fig3}a.
\end{proof}

From now on we will assume that the permutation $\mu$ associated to the
Schubert form $\left(  p/n,q/m,s/l\right)  $ is the product of three disjoint
cycles, $\mu=\tau_{1}\tau_{2}\tau_{3}$. The cyclic decomposition of $\mu$
allows us to determine the number of components of the associated link diagram.

\begin{theorem}
[Classification]\label{teoclaslink}Let $\left(  p/n,q/m,s/l\right)  $ be a
Schubert form, $\gamma$ and $\phi$ its associated permutations given in
\ref{lemmagama} and \ref{lemmafi}, $\mu=\phi\gamma$.\ The $3$-bridge link
diagram $L$ represented by $\left(  p/n,q/m,s/l\right)  $ satisfies:

(i) $L$ is a knot if and only if $a_{p}\notin\mathcal{O}_{\mu}\left(
a_{0}\right)  ,b_{q}\notin\mathcal{O}_{\mu}\left(  b_{0}\right)  $ and
$c_{s}\notin\mathcal{O}_{\mu}\left(  c_{0}\right)  $.

(ii) $L$ is a two component link if and only if one, and only one, of the
following conditions holds: $a_{p}\in\mathcal{O}_{\mu}\left(  a_{0}\right)
,b_{q}\in\mathcal{O}_{\mu}\left(  b_{0}\right)  $ or $c_{s}\in\mathcal{O}%
_{\mu}\left(  c_{0}\right)  $.

(iii) $L$ is a three component link if and only if $a_{p}\in\mathcal{O}_{\mu
}\left(  a_{0}\right)  ,b_{q}\in\mathcal{O}_{\mu}\left(  b_{0}\right)  $ and
$c_{s}\in\mathcal{O}_{\mu}\left(  c_{0}\right)  $.
\end{theorem}

\begin{proof}
Take the cyclic decomposition of$\ \mu=\tau_{1}\tau_{2}\tau_{3}$ and study
each one of the cycles, using the interpretation given in the proof of Theorem
\ref{teored}.
\end{proof}

\section{Orientation of the canonical 3-bridge link diagram $\left(
p/n,q/m,s/l\right)  $\label{secorien}}

Until now we have not considered the orientation of the link $L$, but in order
to find a group presentation for the link group $\pi\left(  L\right)  $ we
will give an orientation to the canonical diagram described in Section
\ref{diagram}. Let $\mu=\phi\gamma=\tau_{1}\tau_{2}\tau_{3}$, we study in
detail these cycles. In each cycle $\tau_{i}$ we have two special vertices,
that are the fixed points of $\gamma$ and form the set $E$ defined in
(\ref{extrem}). Each cycle describes a path around one of the link diagram
underarcs, see Fig. \ref{fig6}.\textit{b}, so one of this special vertices
corresponds to the arc initial point, denoted $I_{i}$; and the other one to
the arc endpoint, denoted $F_{i}$. So we consider that when we follow the
link, we travel it in the order $\tau_{1},\tau_{2}$ and $\tau_{3}$ and the
bridge $a$ \ in the direction from $a_{0}$ to $a_{p}$.

\begin{definition}
We define $\delta_{a}$ (resp. $\delta_{b},\delta_{c}$), the direction in which
we travel the bridge$\ a$ (resp. $b,c$) as: $\delta_{a}=1$ and
\[%
\begin{array}
[c]{ccc}%
\delta_{b}=\left\{
\begin{array}
[c]{cc}%
1, & \text{if we go from }b_{0}\ \text{to }b_{q}\text{\ }\\
-1, & \text{if we go from }b_{q}\ \text{to }b_{0}%
\end{array}
\right.  , & \  & \delta_{c}=\left\{
\begin{array}
[c]{cc}%
1, & \text{if we go from }c_{0}\ \text{to }c_{s}\text{\ }\\
-1, & \text{if we go from }c_{s}\ \text{to }c_{0}%
\end{array}
\right.
\end{array}
\]

\end{definition}

When the condition (i) in Theorem \ref{teoclaslink} is satisfied, the Schubert
form corresponds to a knot diagram, hence the orientation of bridge $a$ is
enough to determine the knot orientation. We take $\tau_{1}\ $as the cycle
that contains $a_{p}$ and $\tau_{3}$ as the cycle that contains $a_{0}$. In
the link case we need to determine the orientation of each component. If $L$
is a 3-component link, the condition (iii) in Theorem \ref{teoclaslink} is
satisfied and we orient each component by $\delta_{a}=\delta_{b}=\delta_{c}%
=1$, $\tau_{1}$ contains $a_{p},\ \tau_{2}$ contains $b_{q}$ and $\ \tau_{3}$
contains $c_{s}$. If $L$ is a 2-component link the condition (ii) in Theorem
\ref{teoclaslink} holds, again we take $\tau_{1}\ $as the cycle that contains
$a_{p}$, and $\tau_{3}$ the cycle that corresponds to the other component.\ 

\begin{lemma}
a. If $L\ $is a knot, Table 1 contains all possibilities for the endpoints of
the cycles $\tau_{1},\tau_{2},\tau_{3}$ and the knot orientation.

b. If $L\ $is a 2-component link, Table 2 contains all possibilities for the
endpoints of the cycles $\tau_{1},\tau_{2},\tau_{3}$ and the link
orientation.
\[%
\begin{array}
[c]{cc}%
\begin{tabular}
[c]{|l|l|l|l|l|l|l|l|}\hline
$I_{1}$ & $F_{1}$ & $I_{2}$ & $F_{2}$ & $I_{3}$ & $F_{3}$ & $\delta_{b}$ &
$\delta_{c}$\\\hline\hline
$a_{p}$ & $b_{0}$ & $b_{q}$ & $c_{0}$ & $c_{s}$ & $a_{0}$ & $1$ & $1$\\\hline
$a_{p}$ & $b_{0}$ & $b_{q}$ & $c_{s}$ & $c_{0}$ & $a_{0}$ & $1$ & $-1$\\\hline
$a_{p}$ & $b_{q}$ & $b_{0}$ & $c_{0}$ & $c_{s}$ & $a_{0}$ & $-1$ & $1$\\\hline
$a_{p}$ & $b_{q}$ & $b_{0}$ & $c_{s}$ & $c_{0}$ & $a_{0}$ & $-1$ &
$-1$\\\hline
$a_{p}$ & $c_{0}$ & $c_{s}$ & $b_{0}$ & $b_{q}$ & $a_{0}$ & $1$ & $1$\\\hline
$a_{p}$ & $c_{0}$ & $c_{s}$ & $b_{q}$ & $b_{0}$ & $a_{0}$ & $-1$ & $1$\\\hline
$a_{p}$ & $c_{s}$ & $c_{0}$ & $b_{0}$ & $b_{q}$ & $a_{0}$ & $1$ & $-1$\\\hline
$a_{p}$ & $c_{s}$ & $c_{0}$ & $b_{q}$ & $b_{0}$ & $a_{0}$ & $-1$ &
$-1$\\\hline
\end{tabular}
&
\begin{tabular}
[c]{|l|l|l|l|l|l|l|l|}\hline
$I_{1}$ & $F_{1}$ & $I_{2}$ & $F_{2}$ & $I_{3}$ & $F_{3}$ & $\delta_{b}$ &
$\delta_{c}$\\\hline\hline
$a_{p}$ & $b_{0}$ & $b_{q}$ & $a_{0}$ & $c_{s}$ & $c_{0}$ & $1$ & $1$\\\hline
$a_{p}$ & $b_{q}$ & $b_{0}$ & $a_{0}$ & $c_{s}$ & $c_{0}$ & $-1$ & $1$\\\hline
$a_{p}$ & $c_{0}$ & $c_{s}$ & $a_{0}$ & $b_{q}$ & $b_{0}$ & $1$ & $1$\\\hline
$a_{p}$ & $c_{s}$ & $c_{0}$ & $a_{0}$ & $b_{q}$ & $b_{0}$ & $1$ & $-1$\\\hline
$a_{p}$ & $a_{0}$ & $b_{q}$ & $c_{0}$ & $c_{s}$ & $b_{0}$ & $1$ & $1$\\\hline
$a_{p}$ & $a_{0}$ & $b_{q}$ & $c_{s}$ & $c_{0}$ & $b_{0}$ & $1$ & $-1$\\\hline
\end{tabular}
\\
\text{ Table 1\ } & \text{Table 2}%
\end{array}
\]

\end{lemma}

To avoid the lack of uniqueness in the cycles, we always write the cycle
$\tau_{i}$ as an ordered set with initial point $I_{i}$, but to simplify
notation we keep the cycle notation. In general this will not generate
confusion in our work.

\begin{lemma}
\label{lemcycle}Let $\mu$ be the permutation associated to the Schubert form
$(p/n,q/m,s/l)$, $\mu=\tau_{1}\tau_{2}\tau_{3}$, for i=1,2,3 we have:

(i) Each cycle $\tau_{i}$ is even, with order $\left\vert \tau_{i}\right\vert
\ $greater than 4.

(ii) $\tau_{i}=\left\{  I_{i},z_{1},\cdots,z_{k},F_{i},\gamma\left(
z_{k}\right)  ,\cdots,\gamma\left(  z_{1}\right)  \right\}  ,\ \ $for
$z_{j}\in A\cup B\cup C,\ \ j=1,\cdots,k,\ \ k\geq1.$

(iii) $\tau_{i}^{\left\vert \tau_{i}\right\vert /2}\ $contains the
transposition $\left(  I_{i},F_{i}\right)  $.
\end{lemma}

\begin{proof}
By condition (\ref{forbiden}) we get $\mu\left(  I_{i}\right)  =\tau
_{i}\left(  I_{i}\right)  \neq F_{i}$, so the length of the cycle $\tau_{i}$
is greater than 3. Then, there exists $z_{j}\in A\cup B\cup C,$ $j=1,\cdots
,k\geq1$ such that $z_{1}=\mu\left(  I_{i}\right)  ,z_{2}=\mu\left(
z_{1}\right)  ,\cdots,z_{k}=\mu\left(  z_{k-1}\right)  $\ and $F_{i}%
=\mu\left(  z_{k}\right)  =\phi\gamma\left(  z_{k}\right)  ,$ this yields%
\[
\mu\left(  F_{i}\right)  =\phi\gamma\left(  F_{i}\right)  =\phi\left(
F_{i}\right)  =\phi\left(  \phi\gamma\left(  z_{k}\right)  \right)
=\gamma\left(  z_{k}\right)  .
\]
Now,%
\[
\mu\left(  \gamma\left(  z_{k}\right)  \right)  =\phi\gamma\left(
\gamma\left(  z_{k}\right)  \right)  =\phi\left(  z_{k}\right)  =\phi\left(
\phi\gamma\left(  z_{k-1}\right)  \right)  =\gamma\left(  z_{k-1}\right)  ,
\]
and then, for $j=k,\cdots,2,$ we get $\mu\left(  \gamma\left(  z_{j}\right)
\right)  =\phi\gamma\left(  \gamma\left(  z_{j}\right)  \right)  =\phi\left(
z_{j}\right)  =\phi\left(  \phi\gamma\left(  z_{j-1}\right)  \right)
=\gamma\left(  z_{j-1}\right)  .$
\end{proof}

The relevant information on each cycle $\tau_{i}$ is contained in the first
part of the cycle, we define the initial segment of $\tau_{i}$ as%
\begin{equation}
\widetilde{\tau_{i}}=\left\{  I_{i},z_{1},\cdots,z_{k}\right\}
\label{defonda}%
\end{equation}

We may summarize the results up to now in an algorithm that allows us to find
the cycles $\tau_{1},\tau_{2},\tau_{3}$, the set $E$ and therefore the
directions $\delta_{b}$ and $\delta_{c}$ associated to a Schubert form.

Let $\sigma$ be the permutation in $A\cup B\cup C$ defined by
\[
\sigma=\left(  a_{0}\ a_{p}\right)  \left(  b_{0}\ b_{q}\right)  \left(
c_{0}\ c_{s}\right)  .
\]
This permutation corresponds to \textit{"travel the bridges"} in the diagram.

\begin{algorithm}
\label{algorithm}Given a Schubert form $\left(  p/n,q/m,s/l\right)  $ the
following algorithm finds the orientation of the associated link diagram. It
provides the cycles $\tau_{1},\tau_{2},\tau_{3}$ such that $\mu=\tau_{1}%
\tau_{2}\tau_{3},$ where the cycle $\tau_{i}$ has the form given in Lemma
\ref{lemcycle}.

1. Take $I_{1}=a_{p},\tau_{1}=\mathcal{O}_{\mu}\left(  I_{1}\right)  $ and
$F_{1}=\tau_{1}^{|\tau_{1}|/2}\left(  I_{1}\right)  .$

2. If $F_{1}=I_{1}$ take $I_{2}=b_{q}$ else take $I_{2}=\sigma\left(
F_{1}\right)  .$

3. Take $\tau_{2}=\mathcal{O}_{\mu}\left(  I_{2}\right)  $ and $F_{2}=\tau
_{2}^{|\tau_{2}|/2}\left(  I_{2}\right)  .$

4. If $\sigma\left(  F_{2}\right)  \notin\left\{  I_{1},F_{1},I_{2}%
,F_{2}\right\}  $ then take $I_{3}=\sigma\left(  F_{2}\right)  $ else take
$I_{3}$ as the unique element in $\left\{  b_{q},c_{s}\right\}  -\left\{
I_{1},F_{1},I_{2},F_{2}\right\}  $.

5. Take $\tau_{3}=\mathcal{O}_{\mu}\left(  I_{3}\right)  $ and $F_{3}=\tau
_{3}^{|\tau_{3}|/2}\left(  I_{3}\right)  $.
\end{algorithm}

\section{Presentation of the 3-bridge link group\label{secgroup}}

Let $L$ be the link diagram with Schubert form $\left(  p/n,q/m,s/l\right)  $.
We have an explicit way to find the over and under presentations of the link
group of $L$, see \cite{BuZi} and \cite{CrFo}. This method requires to use the
link diagram. We will use this method, but we will replace the explicit use of
the diagram by an algorithm that uses the permutations $\phi,\gamma$ and $\mu$
and some new functions defined on the set $A\cup B\cup C$. As the description
of the over and under presentations requires an oriented link diagram, we will
always refer to the standard link diagram and orientation described in Section
\ref{diagram}. It is important to remark that we need the diagram only to
explain the construction, but the algorithm to find the presentations do not
require to draw the link diagram, it depends only on the permutations
$\phi,\gamma$ and $\mu=\phi\gamma$. As $\phi,\gamma$ depend only of the
Schubert form $\left(  p/n,q/m,s/l\right)  $, the presentation of the link
group will depend only on the integers $\left\{  p,n,q,m,s,l\right\}  $. The
algorithm is efficient and easy to implement in a software such as
\textit{Mathematica.}

\subsection{Over presentation of the 3-bridge link $\left(
p/n,q/m,s/l\right)  $}

We take meridians around the bridges as group generators, and label them by
the same name as the bridges, so we have generators $a,b$ and $c$, see Fig.
\ref{fig6}\textit{a.}%

\begin{figure}
[h]
\begin{center}
\includegraphics[
height=1.388in,
width=4.8395in
]%
{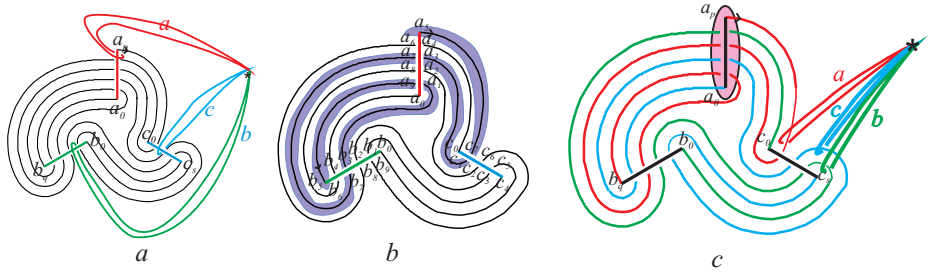}%
\caption{Generators for the over presentation in \textit{a. }and the under
presentation in \textit{c. }Path around an underarc in \textit{b}.}%
\label{fig6}%
\end{center}
\end{figure}

We find the relators by traveling the frontier of a neighborhood of the
underarcs, as shown in Fig. \ref{fig6}\textit{b}, so these paths are precisely
the orbits $\tau_{i},i=1,2,3$.

Each relator is a word in $a,b$ and $c$ constructed with the convention that
each time we cross the bridge $a$ (resp. $b,c$) we write $a^{\pm1}$ (resp.
$b^{\pm1},c^{\pm1}$) depending of the sign of the crossing, given by the
convention {\includegraphics[
height=0.2776in,
width=0.8164in
]%
{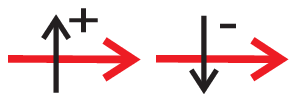}%
}
.

We replace this graphic process by defining a function $\Gamma$ that
"\textit{forgets the index but remembers the direction}". Consider
$\Gamma:A\cup B\cup C\rightarrow\left\{  a^{\pm1},b^{\pm1},b^{\pm1}\right\}  $
defined by%
\begin{align}
\Gamma\left(  a_{i}\right)   &  =\left\{
\begin{array}
[c]{cc}%
a & \text{if }0<i\leq p\\
a^{-1} & \text{otherwise}%
\end{array}
\right.  ,\Gamma\left(  b_{i}\right)  =\left\{
\begin{array}
[c]{cc}%
b^{\delta_{b}} & \text{if }0<i\leq q\\
b^{-\delta_{b}} & \text{otherwise.}%
\end{array}
\right.  ,\label{gamrel}\\
\Gamma\left(  c_{i}\right)   &  =\left\{
\begin{array}
[c]{cc}%
c^{\delta_{c}} & \text{if }0<i\leq s\\
c^{-\delta_{c}} & \text{otherwise.}%
\end{array}
\right.  \nonumber
\end{align}

The relators are $r_{1}=\Gamma\left(  \tau_{1}\right)  ,r_{2}=\Gamma\left(
\mathcal{\tau}_{2}\right)  $ and $r_{3}=\Gamma\left(  \mathcal{\tau}%
_{3}\right)  $, where $\Gamma\left(  \tau_{i}\right)  $ means the word
obtained when we apply $\Gamma$ to each element in the orbit $\tau_{i}$.

Thus we have proved the following proposition.

\begin{proposition}
\label{presover}The link group of the link $L$ given by the Schubert form
$\left(  p/n,q/m,s/l\right)  $ admits a presentation given by
\[
\pi\left(  L\right)  =\left\langle a,b,c\ |\ \Gamma\left(  \tau_{1}\right)
,\Gamma\left(  \tau_{2}\right)  ,\Gamma\left(  \tau_{3}\right)  \right\rangle
\]
were $\mu=\tau_{1}\tau_{2}\tau_{3}$ is the associated permutation and $\Gamma$
is given in (\ref{gamrel}).
\end{proposition}

By the symmetry of the cycles described in Lemma \ref{lemcycle}, we may
rewrite the relators as the relations. When the Schubert form $\left(
p/n,q/m,s/l\right)  $ defines a knot, using the information in Table 1 we find
that
\[
r_{1}:aw_{a}=w_{a}b,r_{2}:bw_{b}=w_{b}c,r_{3}:cw_{c}=w_{c}a,\ \text{in the
first four cases,}%
\]
or%
\[
r_{1}:aw_{a}=w_{a}c,r_{2}:cw_{c}=w_{c}b,r_{3}:bw_{b}=w_{b}a\text{, in the last
four cases.}%
\]
For the case when it is a link we have similar relations.

At this moment we have lost the geometrical meaning of the generators, so we
may rename the generators, if necessary, and unify the two cases, so we have
the following proposition.

\begin{proposition}
\label{presoverpal}The link $L$ given by the Schubert form $\left(
p/n,q/m,s/l\right)  $ admits a presentation given by
\begin{align*}
&  \text{i. }\left\langle a,b,c\ |\ aw_{1}=w_{1}b,\ \ \ \ bw_{2}%
=w_{2}c,\ \ \ \ \ cw_{3}=w_{3}a\right\rangle \text{ if }L\text{ is a knot,}\\
&  \text{ii. }\left\langle a,b,c\ |\ aw_{1}=w_{1}b,\ \ \ \ bw_{2}%
=w_{2}a,\ \ \ \ \ cw_{3}=w_{3}c\right\rangle \text{ if }L\ \text{is a
2-component link,}\\
&  \text{iii. }\left\langle a,b,c\ |\ aw_{1}=w_{1}a,\ \ \ \ bw_{2}%
=w_{2}b,\ \ \ \ \ cw_{3}=w_{3}c\right\rangle \text{ if }L\ \text{is a
3-component link,}%
\end{align*}
were $w_{i}$ is a word in $a,b,c$ given by$\ w_{i}=\Gamma\left(
\widetilde{\tau_{i}}\right)  $ were $\Gamma$ is defined in (\ref{gamrel}) and
$\widetilde{\tau_{i}}$ is defined in (\ref{defonda}).
\end{proposition}

We know that in the case of knots, one of the relations is redundant, but for
practical reasons we prefer to have all of them and, in particular
computations, we omit the longest relation, that we do not know in advance
which one will be.$\ $This contrasts with the under presentation that we will
introduce in the next section, in which we know the lengths of the words involved.

\begin{lemma}
The sum of the lengths of the words $w_{1},w_{2}$ and $w_{3}$ is $p+q+s-3.$
\end{lemma}

\begin{lemma}
When $L$ is a knot, the peripheral system of the group is given by
$\left\langle a,l\right\rangle $ with $l=w_{1}w_{2}w_{3}a^{-k}$ and $k$ is the
exponent sum of the word $w_{1}w_{2}w_{3}$.
\end{lemma}

\begin{proof}
By direct computation we have
\[
al=aw_{1}w_{2}w_{3}a^{-k}=w_{1}bw_{2}w_{3}a^{-k}=w_{1}w_{2}cw_{3}a^{-k}%
=w_{1}w_{2}w_{3}aa^{-k}=la
\]

\end{proof}

\subsection{Under presentation of the 3-bridge link $\left(
p/n,q/m,s/l\right)  $}

The under presentation is the dual presentation of the over presentation.
Those dual presentations play a central role in the proofs of properties of
the knot group and the Alexander polynomial of knots, see \cite{CrFo}. By
studying these presentations we find a similar algorithm to the known one to
find the 2-bridge link group, that has an explicit formula depending of $p$
and $q$. Of course, we need a more elaborate algorithm.

We take as generators of $\pi\left(  L\right)  $ the meridians around the
underarcs, see \ref{fig5}\textit{c. }Again, the key point is to use the cyclic
decomposition of $\mu$. We call $a$ (resp. $b$ and $c)$ the generators
corresponding to the underarc described by $\tau_{1}$ (resp. $\tau_{2}$ and
$\tau_{3}$).

The relations are given by traveling the boundary of each butterfly, that
describe simple closed paths around the bridges.

So the first path is given by $\left\{  a_{0},\cdots,a_{2p-1}\right\}  $, the
second by $\left\{  b_{0},\cdots,a_{2q-1}\right\}  $ and the third by
$\left\{  c_{0},\cdots,c_{2s-1}\right\}  $. The graphical procedure to find
the relators is: Each time we cross the link we encounter one of the vertices
in the set $A\cup B\cup C,$ we identify the underarc, say $x$, and the sign of
the crossing, $sg,$ and write $x^{sg}$, with $x\in\left\{  a,b,c\right\}  $
and $sg=\pm1$.

Again, this procedure will be established by defining a function$\ \rho$,
similar to the one defined in (\ref{gamrel}), that identifies the underarc that
contains the vertex and the direction of the crossing. Let $\rho:A\cup B\cup
C\rightarrow\left\{  a^{\pm1},b^{\pm1},c^{\pm1}\right\}  $ defined by
\[
\text{If }x\in E,\ \rho\left(  x\right)  =\left\{
\begin{array}
[c]{ll}%
a & \text{if }x\in\widetilde{\tau_{1}}\\
a^{-1} & \text{if }x\notin\widetilde{\tau_{1}}\ \\
b^{\delta_{b}} & \text{if }x\in\widetilde{\tau_{2}}\\
b^{-\delta_{b}} & \text{if }x\notin\widetilde{\tau_{2}}\ \\
c^{\delta_{c}} & \text{if }x\in\widetilde{\tau_{3}}\\
c^{-\delta_{s}} & \text{if }x\notin\widetilde{\tau_{3}}\
\end{array}
\right.  \text{.\ \ \ If\ }x\notin E,\ \rho\left(  x\right)  =\left\{
\begin{array}
[c]{ll}%
a & \text{if }x\in\widetilde{\tau_{1}}\\
a^{-1} & \text{if }\gamma\left(  x\right)  \in\widetilde{\tau_{1}}\ \\
b^{\delta_{b}} & \text{if }x\in\widetilde{\tau_{2}}\\
b^{-\delta_{b}} & \text{if }\gamma\left(  x\right)  \in\widetilde{\tau_{2}%
}\ \\
c^{\delta_{c}} & \text{if }x\in\widetilde{\tau_{3}}\\
c^{-\delta_{s}} & \text{if }\gamma\left(  x\right)  \in\widetilde{\tau_{3}}\
\end{array}
\right.
\]
\label{defro}

The relators are
\begin{align*}
s_{a}  &  =\rho\left(  a_{0}a_{1}\cdots a_{2p-1}\right)  =\rho\left(
a_{0}\right)  \rho\left(  a_{1}\right)  \cdots\rho\left(  a_{2p-1}\right)  ,\\
s_{b}  &  =\rho\left(  b_{0}b_{1}\cdots b_{2p-1}\right)  =\rho\left(
b_{0}\right)  \rho\left(  b_{1}\right)  \cdots\rho\left(  b_{2q-1}\right)  ,\\
s_{c}  &  =\rho\left(  b_{0}b_{1}\cdots b_{2p-1}\right)  =\rho\left(
c_{0}\right)  \rho\left(  c_{1}\right)  \cdots\rho\left(  c_{2s-1}\right)  .
\end{align*}

For the symmetry of the functions and the cycles given in Lemma \ref{lemcycle}%
, we have that if $\gamma\left(  x\right)  \neq x,$ $\rho\left(  \gamma\left(
x\right)  \right)  =\rho\left(  x\right)  ^{-1}$, therefore if we take the
words%
\begin{equation}
u_{a}=\rho\left(  a_{1}\right)  \cdots\rho\left(  a_{p-1}\right)
,\ \ \ u_{b}=\rho\left(  b_{1}\right)  \cdots\rho\left(  b_{q-1}\right)
,\ \ \ \ \ u_{c}=\rho\left(  c_{1}\right)  \cdots\rho\left(  c_{s-1}\right)  ,
\label{defpalu}%
\end{equation}
the relators become the relations
\[
cu_{a}=u_{a}a,\ \ \ \ \ \ \ \ \ \ au_{b}=u_{b}b,\ \ \ \ \ \ bu_{c}=u_{c}c
\]
or%
\[
bu_{a}=u_{a}a,\ \ \ \ \ \ \ \ \ \ au_{c}=u_{c}c,\ \ \ \ \ \ cu_{b}=u_{b}a.
\]
Note that the lengths of the words $u_{a},u_{b}$ and $u_{c}$ are $p-1,$ $q-1$
and $s-1$,$\ $respectively. In this case it is not possible to change the
variable names, because we want to have the information about the word lengths.

Now the peripheral system is given by $\left\langle a,l^{\prime}\right\rangle
$ where $l^{\prime}=u_{a}u_{b}u_{c}f^{-e}$, were $e$ is the exponent sum of
the word $u_{a}u_{b}u_{c}$. We have proved the following theorem:

\begin{theorem}
\label{presunder}The link $L$ given by the butterfly $\left(
p/n,q/m,s/l\right)  $ admits a presentation given by:

i. If $L$ is a knot
\begin{align*}
&  \left\langle a,b,c\ |\ cu_{a}=u_{a}a,\ \ \ \ \ \ \ \ \ \ au_{b}%
=u_{b}b,\ \ \ \ \ \ bu_{c}=u_{c}c\right\rangle ,\ \ \ \ \ \ \ \ \text{or}\\
&  \left\langle a,b,c\ |\ bu_{a}=u_{a}a,\ \ \ \ \ \ \ \ \ \ au_{c}%
=u_{c}c,\ \ \ \ \ \ cu_{b}=u_{b}a.\right\rangle ,
\end{align*}

ii. If $L$ is a 3-component link
\[
\left\langle a,b,c\ |\ au_{a}=u_{a}a,\ \ \ \ \ \ \ \ \ \ bu_{b}=u_{b}%
b,\ \ \ \ \ \ cu_{c}=u_{c}c\right\rangle ,
\]
with $u_{a},u_{b}$ and $u_{c}$ words of length $p-1,$ $q-1$ and $s-1$%
,$\ $respectively, defined by (\ref{defpalu}).
\end{theorem}

If $L$ is a 2-component link there are six possible combination for the
presentation, that are the natural variations of $\left\langle
a,b,c\ |\ au_{a}=u_{a}a,\ cu_{b}=u_{b}b,\ bu_{c}=u_{c}c\right\rangle .$

\textbf{Note:} This construction does not depend on the fact that $p\geq q\geq
s$, nor that we are working with type I butterfly. So we may use it in a more
general way. However, if we take the Schubert form $\left(
p/n,q/m,s/l\right)  $ we know that in the knot case one of the relations is
redundant, and in this presentation we know that the longest is the first one,
so usually that is the one we eliminate.

\section{Special family: $(p/n,p/n,p/n)$}

In general we do not have an exact pattern for a 3-bridge link group, as the
one we encounter for 2-bridge links, see \cite{Mur}, but there are families of
3-bridge links with a very regular pattern for the fundamental group. One of
them is the family of links with Schubert form $(p/n,p/n,p/n)$, for integers
$1\leq n\leq p$. This family contains: Borromean rings $\left(
5/2,5/2,5/2\right)  $, the pretzel link $P\left(  p,p,p\right)  $, that
corresponds to $\left(  2p/p,2p/p,2p/p\right)  $; the toroidal knot $T\left(
3,p\right)  $ and its mirror image $T\left(  3,-p\right)  $, that corresponds
to $(p/1,p/1,p/1)$ and to $(p/p,p/p,p/p)$, respectively. The standard diagrams
of the links in this family have symmetries of order 2 and 3.

\begin{proposition}
\label{proppalw}For the link $L$ with Schubert form $(p/n,p/n,p/n)$ there
exists a word $w\left(  x,y,z\right)  $ in the variables $x,y$ and $z,$ such
that if $w_{a}=w\left(  a,b,c\right)  ,w_{b}=w\left(  b,c,a\right)  $ and
$w_{c}=w\left(  c,a,b\right)  $ then:

i. If ~$L$ is a knot, the knot group has the presentation
\[
\left\langle a,b,c\ |\ aw_{a}=w_{a}b,bw_{b}=w_{b}c,cw_{c}=w_{c}a\right\rangle
\text{.}%
\]

2. If $L$ is a link, it has 3 components and the link group has the
presentation
\[
\left\langle a,b,c\ |\ aw_{a}=w_{a}a,bw_{b}=w_{b}b,cw_{c}=w_{c}c\right\rangle
\text{.}%
\]

\end{proposition}

\begin{proof}
Study the symmetry of the link diagram.
\end{proof}

\begin{example}
1. The Borromean rings have Schubert normal form $\left(  5/2,5/2,5/2\right)
$ and $w\left(  x,y,z\right)  =yz^{-1}y^{-1}z$.

2. The knot $8_{19}\ $in Rolfsen\'{}s
 table has Schubert normal form$\ \left(  4/1,4/1,4/1\right)  $ and $w\left(
x,y,z\right)  =zyx$. Note that it is the toroidal knot $T\left(  3,4\right)
$. In general, the toroidal link $\left(  p/1,p/1,p/1\right)  $ is a
3-component link if $p\equiv1\ \operatorname{mod}3$ and it is a knot in the
other cases; and the word $w$ in Proposition \ref{proppalw} is
\[
w\left(  x,y,z\right)  =\left\{
\begin{array}
[c]{cc}%
\left(  zyx\right)  ^{p/3} & \text{If }p\equiv1\ \operatorname{mod}3\\
\left(  zyx\right)  ^{[p/3]}z & \text{If }p\equiv2\ \operatorname{mod}3\\
\left(  zyx\right)  ^{[p/3]}zy & \text{If }p\equiv0\ \operatorname{mod}3,
\end{array}
\right.
\]
were $\left[  m\right]  $ means the integer part of $m$.

3. The knot $9_{35}\ $in Rolfsen\'{}s
 table has Schubert normal form$\ \left(  6/3,6/3,6/3\right)  $ and $w\left(
x,y,z\right)  =z^{-1}xz^{-1}yz^{-1}$. Note that it is the Pretzel knot
$\left(  3,3,3\right)  $. In general the Pretzel link $\left(
2p/p,2p/p,2p/p\right)  $ is a knot if $p$ is odd and a 3-component link if $p$
is even and the word $w$ in Proposition \ref{proppalw} is
\[
w\left(  x,y,z\right)  =\left\{
\begin{array}
[c]{cc}%
\left(  z^{-1}x\right)  ^{\left[  p/2\right]  }z^{-1}\left(  z^{-1}y\right)
^{\left[  p/2\right]  } & \text{If }p\ \text{is odd}\\
\left(  z^{-1}x\right)  ^{(p-2)/2}z^{-1}xy^{-1}\left(  xy^{-1}\right)
^{\left(  p-2\right)  /2} & \text{If }p\ \text{is even.}%
\end{array}
\right.
\]

\end{example}

\end{document}